\documentclass[reqno,a4paper]{article}
\usepackage{amsmath,amssymb,amsfonts,epsfig,float,color,cite,algorithm,algorithmic,fancyhdr}
\definecolor{lightblue}{rgb}{0.22,0.45,0.70}
\definecolor{cgray}{rgb}{0.9,0.9,0.9}
\usepackage[colorlinks=true,breaklinks=true,linkcolor=lightblue,citecolor=lightblue]{hyperref}

\numberwithin{equation}{section}

\renewcommand{\div}{\operatorname*{div}}
\newcommand{\norm}[1]{\ensuremath{\left\|#1\right\|}}
\newcommand*\nnorm[1]{\left|\!\left|\!\left|#1\right|\!\right|\!\right|}
\newcommand*\nnormh[1]{\left|\!\left|\!\left|#1\right|\!\right|\!\right|_h}
\newcommand{\jump}[1]{\ensuremath{[\![#1]\!]}}

\newcommand\Om{\Omega}
\newcommand\Qad{Q_{\mathrm{ad}}}

\newcommand\bn{\boldsymbol{n}}
\newcommand\be{\boldsymbol{e}}
\newcommand\bu{\boldsymbol{u}}
\newcommand\bx{\boldsymbol{x}}
\newcommand\bU{\mathbf{U}}
\newcommand\bA{\mathbf{A}}
\newcommand\bB{\mathbf{B}}
\newcommand\bF{\mathbf{F}}
\newcommand\bH{\mathbf{H}}
\newcommand\bD{\mathbf{D}}
\newcommand\bE{\mathbf{E}}
\newcommand\bW{\mathbf{W}}
\newcommand\bR{\mathbf{R}}
\newcommand\bG{\mathbf{G}}
\newcommand\bS{\mathbf{S}}
\newcommand\bZ{\mathbf{Z}}
\newcommand\bPhi{\boldsymbol{\Phi}}
\newcommand\balpha{\boldsymbol{\alpha}}
\newcommand\bdelta{\boldsymbol{\delta}}
\newcommand\bbeta{\boldsymbol{\beta}}
\newcommand\bv{\boldsymbol{v}}
\newcommand\cT{\mathcal{T}}
\newcommand{\dx}{\,\mathrm{d} \bx}
\newcommand{\dt}{\,\mathrm{d}t}
\newcommand{\ds}{\,\mathrm{d}s}
\newtheorem{theorem}{Theorem}[section]

\newtheorem{lemma}[theorem]{Lemma}
\newtheorem{definition}{Definition}[section]
\newtheorem{asmp}{Assumption}[section]

\newenvironment{proof}{\noindent{\it Proof.}}{\hfill$\square$}

\allowdisplaybreaks

\title{Mixed and discontinuous
finite volume element schemes for the optimal control of immiscible flow in porous media\thanks{This work
has been partially supported by the EPSRC through the Research Grant EP/R00207X/1. }}
\author{ {\sc Sarvesh Kumar}\thanks{Department of Mathematics, Indian Institute of Space
Science and Technology, Thiruvananthapuram 695 547, Kerala, India.
Email: \texttt{sarvesh@iist.ac.in}.},\quad
{\sc Ricardo Ruiz Baier}\thanks{Mathematical Institute,
University of Oxford, A. Wiles Building,
Woodstock Road, Oxford OX2 6GG, UK. E-mail: {\tt ruizbaier@maths.ox.ac.uk}.}, \quad
{\sc Ruchi Sandilya}\thanks{Centre For Applicable Mathematics,
Tata Institute of Fundamental Research, Bangalore 560065, India.
E-mail: \texttt{ruchi@tifrbng.res.in}.}}

\date{}
\begin{document}
\maketitle
\begin{abstract}
In this article we introduce a family of hybrid discretisations for
the numerical approximation of optimal control problems governed by
the equations of immiscible displacement in porous media. The proposed
schemes are based on mixed and discontinuous finite volume
element methods in combination with the optimise-then-discretise
approach for the approximation of the optimal control problem, leading
to nonsymmetric algebraic systems, and employing minimum
regularity requirements. Estimates for the error (between a
local reference solution of the infinite dimensional optimal control
problem and its hybrid approximation) measured in suitable norms are
derived, showing optimal orders of convergence.
\end{abstract}

\vspace{5pt}

\vspace{5pt}

\section{Introduction}

\paragraph{Scope.}
We are interested in the accurate representation of the flow patterns produced by
immiscible fluids within porous media. With the growing importance of the underlying
physical processes in a variety of applications, the mathematical models used to
describe this scenario have received a considerable attention in the past few decades.
A popular example can be encountered in petroleum engineering, specifically in the
standard process of oil recovery. The strategy there consists in injecting water (or other
fluids having favourable density and viscosity properties) in such a way that the
 oil trapped in subsurface reservoirs is displaced mainly by pressure gradients. In its
 classical configuration, the technique of oil recovery by water injection employs two
 wells that contribute to maintain a high pressure and adequate flow rate in
 the oil field: an injection well from where the non-oleic liquid is injected, pushing the
 remaining oil towards a second, production well, from which oil is transported to the surface.

Regarding the simulation of these processes using mathematical models and numerical methods,
there is a rich body of literature dealing with mixed finite element (FE) formulations
where the filtration velocity and the pressure of each phase are solved at once (see, for instance, the
classical works \cite{darlow84,douglas183,douglas283,ewing80}). Mixed methods constructed
using $H(div)-$conforming elements for the flux variable also allow for local mass conservation.
Alternative methods, also widely used in a variety of different formulations, include discontinuous
Galerkin (DG) schemes which do not require inter-element continuity and feature element-wise
conservation, arbitrary accuracy, controlled numerical diffusion, and can handle more adequately
problems with rough coefficients (see, for instance, \cite{arnold02} for a general overview on DG
methods and \cite{sun02, sun105, sun05, whiteley15} for their application in different configurations
of multiphase flows).

A recurrent strategy in the design of numerical methods for coupled flow-transport problems as the
one described above, is to combine different techniques with the objective of retaining the main
properties of each compartmental scheme. For example, combined mixed FE and DG methods
have been applied in \cite{bartels09,li15,sun02} to numerically solve the coupled system of
miscible displacement in porous media. On the other hand, a mixed finite volume element (FVE) method
approximating the velocity-pressure pair and a discontinuous finite volume element (DFVE) scheme for
the saturation equation are combined in \cite{kumar12}.
FVE schemes require to define trial and test spaces associated to primal and dual partitions of
the domain, respectively. Different types of dual meshes are employed when the FVE method
is of conforming, non-nonconforming, or discontinuous type (see details and comparisons in
e.g. \cite{carstensen16, chou07, cui10}), but in most cases they feature local conservativity
as well as suitability for deriving $L^2-$error estimates. We point out that schemes belonging
to the particular class of
DFVE approximations preserve features of both DG
and general FVE methods, including smaller support of dual elements (when compared with
conforming and non-conforming FVEs) and appropriateness in handling discontinuous coefficients.

Also in the context of FVE methods, the development in \cite{brt12} uses a mixed
(or hybrid) conforming-nonconforming discretisation applied to sedimentation problems, \cite{bkr15,bkkr16} analyse DFVE methods
applied to viscous flow and degenerate parabolic equations, and \cite{ruiz16} introduces mixed FE in
combination with DFVE for a general class of multiphase problems. An extensive survey on different
methods  for multiphase multicomponent flows in porous media can be found in \cite{chen06,helmig06,huber99}.

 \paragraph{Optimal control and immiscible flow in porous media.}
 Oil recovery in its so-called primary and secondary stages, can only lead to the extraction of
 20\%-40\% of the reservoir's original oil. Other techniques (including a tertiary stage and the enhanced
 oil recovery process) can increase these numbers up to 30\%-60\%, but the development of control
 devices for manipulating the progression of the oil-water front, therefore increasing further the oil recovery,
 is still a topic of high interest. A viable approach consists in solving optimal control problems subject to
 the equations of two-phase incompressible immiscible flow in porous media. The goal is quite clear: to achieve
 optimal oil recovery from underground reservoirs after a fixed time interval. Several variables enter into
 consideration (as the price of oil and water, rock porosity and intrinsic permeability, the mobilities of the fluids,
 the constitutive relations defining capillary pressure, and so on) but here we will restrict the study  to
 the adjustment of the water injection only.

 Control theory and adjoint-based methods have
 been exploited in the optimisation of several aspects of the process, for instance in the design of valve operations for wells
 (see e.g. \cite{mehos89,brouwer04} and the review paper \cite{jansen11}).
 However, and in contrast with the situation observed for the approximation of direct systems, the numerical \emph{analysis}
 of optimal control problems governed by incompressible flows in porous media (meaning
 rigorous error estimates and stability properties) has been so far restricted  to classical
discretisations.  These include the FE  method for immiscible displacement optimal control studied in
\cite{chang14} and the box method for the constrained optimal control problems with partially miscible
two phase flow in porous media considered in \cite{simon13}. Our goal here is to investigate optimal control
problems governed by two-phase incompressible immiscible  flow in porous media and their discretisation
using a  combined mixed FVE discretisation for the flow equations, and a DFVE scheme for the approximation of
the transport equation. We concentrate our development on the optimise-then-discretise  approach, where one
first formulates the continuous optimality conditions and then the discretisation is applied to the continuous
optimal system (see its applicability in similar scenarios in e.g. \cite{collis02,luo13}).

\paragraph{Outline.}
The remainder of the paper is organised as follows. In Section~\ref{sec:model} we state the model problem
together with the corresponding optimality conditions, and present some preliminary results.
Section \ref{sec:fve} provides details about the discrete formulation, starting with the time discretisation
and following with the presentation of our mixed FVE/DFVE scheme applied to the optimal control
problem under consideration. In Section \ref{sec:error} we advocate to the derivation of  a priori error
estimates in suitable norms, whereas Section~\ref{sec:impl} gives an overview of the implementation
strategy employed in the solution of the overall optimal control problem.

\section{Set of governing equations}\label{sec:model}
We consider an optimal control problem governed by a nonlinear coupled system of equations representing the interaction of two incompressible  fluids in a porous structure $ \Omega \subset \mathbb{R}^2 $. We study the process occurring within the time interval $ J=(0,T] $, where the optimisation problem reads
\begin{equation}
\min\limits_{q \in \Qad } \mathcal{J}(q):=\frac{1}{2} \int_{\Omega}\tilde{w}c^2(T)\dx+\frac{\alpha_0}{2}\int_0^T \int_{\Omega}\delta_0 q(t)^2 \dx \dt, \label{obj1}
\end{equation}
subject to
\begin{align}
 \bu   & = - \kappa(\bx)\lambda(c)\nabla p, & \forall (\bx,t)\in \Omega \times J, \nonumber \\
 \nabla\cdot\bu &= (\delta_0-\delta_1)q(t), & \forall (\bx,t)\in \Omega \times J, \label{sys1}\\
 \phi \partial_t c-\nabla\cdot(\kappa(\bx)(\lambda\lambda_o\lambda_w p'_c)(c)\nabla c)+\lambda'_o(c)\bu \cdot \nabla c & =-\lambda_o(c)\delta_0 q(t), &\forall (\bx,t)\in \Omega \times J. \nonumber
\end{align}
Here $ c(\bx,t) $ represents the saturation of oil in the two-phase fluid, $ \phi(\bx) $ the porosity of the rock, $ \kappa(\bx) $ the permeability of the porous rock, $ \lambda(c) $ the total mobility of the two-phase fluid, $ \lambda_o(c) $ the relative mobility of the oil, $ \lambda_w(c) $ the relative mobility of the water, $ \bu(\bx,t) $ the Darcy velocity of the fluid mixture, $ q(t) $ the flow rate, $ p_c(c) $ the capillary pressure, $ \tilde{w} $ the price of oil and $ \alpha_0 $ the price of water. The terms $ \delta_0 $ and $ \delta_1 $ are Dirac functions located at the injection and production wells, respectively. For a given $\hat{q}>0$, by $\Qad  $ we denote the set of admissible controls
$$ \Qad =\lbrace q \in L^\infty[0,T]: 0\leq q \leq \hat{q}\rbrace. $$
The overall mechanism consists in finding a control $ q $ over a time interval $ [0,T] $ that minimises the remaining oil in the reservoir by adjusting the amount of injected water.

For sake of the analysis and discretisation of the problem, we rewrite the system equations in a slightly different notation.
Let us introduce the functions
$$ \alpha(c)=[\kappa(\bx)\lambda(c)]^{-1} ,\quad
 \mathcal{D}(c)=\kappa(\bx)\lambda(c)\lambda_o(c)\lambda_w(c) p'_c(c), \quad b(c)=\lambda'_o(c),  \quad
f(c)=-\lambda_o(c), $$ and let us assume that
$ 0<a_*\leq \alpha^{-1}(c) \leq a^* $, $ \phi_*\leq \phi(\bx) \leq \phi^* $ and $ 0<d_*\leq \mathcal{D}(c) \leq d^* $.
We also assume that $ \alpha(c), b(c), \mathcal{D}(c) $ and $ f(c) $ are Lipschitz continuous functions of $ c $.

The state system \eqref{sys1} is subject to slip velocities and zero-flux boundary conditions for the concentration:
$$
 \bu\cdot\bn =0,\quad \text{and} \quad
 \mathcal{D}(c)\nabla c \cdot \bn = 0, \quad \forall (\bx,t)\in \partial\Omega \times J, $$
 together with a compatibility zero-mean condition for the pressure
 $$ \int_{\Om} p( \bx,t)\dx=0, \quad \forall t\in J, $$
and a suitable initial datum for the saturation
$$c(\bx,0)=c_0(\bx), \quad \forall \bx \in \Omega.
$$

Let the points $ \bx_0 $ and $ \bx_1 $ denote the location of injection and production wells, respectively.
In view of constructing numerical approximations using classical methods, the Dirac delta functions
appearing as source terms in the mass conservation equation of \eqref{sys1} can be regularised
as done in e.g. \cite{chang14}. Let $ \bx_0\in \Omega_0, \bx_1\in \Omega_1 \subset \Omega $,
with $ \Omega_0 \cap \Omega_1=\emptyset $ and $ |\Omega_0|=|\Omega_1|=\sigma $ with $ 0<\sigma \ll 1 $.
We next proceed to define the functions
 \begin{align*}
   r_i=
\begin{cases}
    {1}/{\sigma}, & \, \bx\in \Omega_i\\
    0, & \, \text{otherwise},
\end{cases}\quad i=0,1, \quad \text{and} \quad
  w(\bx,t)=
\begin{cases}
    {\tilde{w}}/{\epsilon}, & \, (\bx,t)\in \Omega\times[T-\epsilon,T],\\
    0, & \, (\bx,t)\in \Omega\times[0,T-\epsilon),
\end{cases}
\end{align*}
for a given $\epsilon>0$. Then we can rewrite the optimal control problem \eqref{obj1}-\eqref{sys1} as follows
\begin{equation}\label{obj2}
\min\limits_{q \in \Qad } \mathcal{J}(q):=\frac{1}{2} \int_0^T \int_{\Omega}w(\bx,t)c^2(\bx,t)\dx\dt+\frac{\alpha_0}{2}\int_0^T q(t)^2 \dt,
\end{equation}
subject to
\begin{align}
 \alpha(c)\bu+\nabla p & = \boldsymbol{0}, &\forall (\bx,t)\in \Omega \times J, \nonumber\\
 \nabla\cdot\bu &= (r_0-r_1)q, &\forall (\bx,t)\in \Omega \times J,  \label{sys2} \\
 \phi \partial_t c-\nabla\cdot(\mathcal{D}(c)\nabla c)+b(c)\bu \cdot \nabla c &=f(c)r_0 q, &\forall (\bx,t)\in \Omega \times J.
\nonumber
\end{align}
We make the following assumptions on the system coefficients (see a similar treatment in e.g. \cite{ewing80}):
\begin{asmp} \label{asmp1}
There exists a uniform constant $ M_0>0 $ such that
\begin{align*}
\norm{\alpha^{-1}(c)}_{L^\infty(J;L^\infty(\Omega))} \leq M_0,\, \norm{b(c)}_{L^\infty(J;L^\infty(\Omega))} \leq M_0,\\
\norm{\mathcal{D}(c)}_{L^\infty(J;L^\infty(\Omega))} \leq M_0,\, \norm{f(c)}_{L^\infty(J;L^\infty(\Omega))} \leq M_0.
\end{align*}
\end{asmp}
Under Assumption \ref{asmp1}, the optimal control problem \eqref{obj2}-\eqref{sys2} admits at least one solution (for details we refer to \cite[Theorem 2.1]{chang14}). However, as the state system comprises  coupled nonlinear PDEs, the optimisation problem is non-convex and hence may exhibit multiple solutions. Therefore, we will assume a local optimal control (see a related strategy in \cite{casas02}) of problem \eqref{obj2}-\eqref{sys2} which satisfies the first order necessary and second order sufficient optimality conditions.
\begin{definition}
A control $ q \in \Qad  $ is said to be a local optimal solution of \eqref{obj2}-\eqref{sys2} in the sense of $ L^2[0,T] $, if there is an $ \epsilon >0 $ such that
$$ \mathcal{J}(q) \leq \mathcal{J}(\tilde{q}) \quad \forall \tilde{q} \in \Qad  \quad \text{with} \quad \norm{\tilde{q}-q}_{L^2[0,T]} \leq \epsilon.$$
\end{definition}
\begin{asmp} \label{asmp2}
There exists  $ M_1 >0$ such that
\begin{align*}
\norm{\bu}_{L^\infty(J;L^\infty(\Omega)^2)} \leq M_1,\, \norm{\nabla c}_{L^\infty(J;L^\infty(\Omega))} \leq M_1,\\
\norm{\mathcal{D}'(c)}_{L^\infty(J;L^\infty(\Omega))} \leq M_1,\, \norm{\alpha'(c)}_{L^\infty(J;L^\infty(\Omega))} \leq M_1.
\end{align*}
\end{asmp}
Assumptions \ref{asmp1} and \ref{asmp2} imply that the local solution $ q $ of \eqref{obj2}-\eqref{sys2} satisfies the classical first order optimality conditions,
which can be formulated as
\begin{align}
\int_0^T (f(c)r_0c^*-(r_0-r_1)p^*+\alpha_0 q, \tilde{q}-q) \dt \geq 0, \quad \forall \tilde{q} \in \Qad , \label{ncd}
\end{align}
where, $ (\bu^*,p^*,c^*) $ is the costate velocity, costate pressure and costate saturation associated with $ q $, and solving the adjoint system (see \cite[Theorem 3.1]{chang14}):
\begin{equation}\begin{split}
 \alpha(c)\bu^*+\nabla p^*+c^*b(c)\nabla c & = \boldsymbol{0},\\
 \nabla\cdot\bu^* & = 0, \\
 -\phi \partial_t c^*-\nabla\cdot(\mathcal{D}(c)\nabla c^*)-(b(c)\bu-\mathcal{D}'(c)\nabla c)\cdot \nabla c^*+\alpha'(c)\bu^*\cdot\bu +r_1qb(c)c^*&= wc,
 \end{split} \label{adjc}
\end{equation}
for a.e. $(\bx,t)\in \Omega \times J$,  associated with boundary conditions:
$$ \bu^*\cdot \bn=0, \qquad \mathcal{D}(c)\nabla c^* \cdot \bn = 0, \qquad \forall (\bx,t)\in \partial\Omega \times J,$$
and final condition $ c^*(\bx,T)=0 $. Finally, as commonly done for nonlinear systems (see e.g.  \cite{casas02,hinze01,neitzel12}), we assume that the local solution $ q $ of \eqref{obj2}-\eqref{sys2} satisfies the following second order sufficient condition:
There exists $ C_0 >0 $ such that
\begin{equation}
\mathcal{J}''(q)(\tilde{q},\tilde{q}) \geq C_0\norm{\tilde{q}}_{L^2[0,T]}^2,\quad \forall \tilde{q} \in \Qad . \label{suf}
\end{equation}
For our forthcoming analysis we recall the definition of the space $ H(\div;\Omega):=\lbrace \bv \in L^2(\Om)^2:\nabla\cdot \bv  \in L^2(\Om) \rbrace $, equipped with the norm
$$
\norm{\bv}_{\div,\Om}^2:=\norm{\bv}_{0,\Om}^2+\norm{\nabla\cdot\bv}_{0,\Om}^2,
$$
where $ \norm{\cdot}_{0,\Om}$ will be employed throughout the text to denote the norm for both the spaces
$ L^2(\Om) $ and for its vectorial counterpart $ L^2(\Om)^2$ .
Then we introduce the admissibility spaces for velocity and pressure
$$ U=\lbrace \bv \in H(\div;\Om): \bv\cdot \bn=0 \,\, \text{on}\,\, \partial \Om \rbrace,
\quad \text{and} \quad W=L^2(\Om)/\mathbb{R},$$
respectively.

\section{Finite dimensional formulation}\label{sec:fve}

\paragraph{Spatial discretisation.}
The velocity-pressure equations involved in the state and costate systems will be discretised via  mixed FVE, whereas the saturation equation will follow a DFVE formulation. In turn, the approximation of the control variable will be carried out using a variational method (see \cite{hinze05}), where the control set is discretised by a projection of the discrete costate variables. Based on a first primal partition of the domain, we will require two additional
dual meshes where the mixed and discontinuous FVE approximations will be defined.

Let us consider a regular, quasi-uniform partition $ \lbrace \mathcal{T}_h \rbrace_{h>0} $ of $ \bar{\Omega} $ into  triangles $ K $, of maximum diameter $ h $. Let $ e $ be an interior edge shared by two elements $ { K}_1 $ and $ { K}_2$  in $\cT_h $ with outward unit normal vectors $\mathbf{ n_1 } $ and $\mathbf{ n_2 } $, respectively. For a generic scalar $ q$, let $ \jump{q}:={q}|_{\partial { K}_1}\mathbf{ n_1 }+{q}|_{\partial { K}_2}\mathbf{ n_2 } $ and $ \langle q\rangle := \frac{1}{2}({q}|_{\partial { K}_1}+{q}|_{\partial { K}_2}) $ denote its jump and average value on $ e $. For a generic vector $ \mathbf{r} $, its jump and average across edge  $ e $ is denoted respectively, by $ \jump{\mathbf{r}}:=\mathbf{\mathbf{r}|_{\partial { K}_1}}\cdot \mathbf{ n_1 }+\mathbf{\mathbf{r}|_{\partial { K}_2}}\cdot\mathbf{ n_2 } $ and $ \langle \mathbf{r}\rangle := \frac{1}{2}(\mathbf{\mathbf{r}|_{\partial { K}_1}}+\mathbf{\mathbf{r}|_{\partial { K}_2}}) $. For a boundary edge $ e $ with  outward normal $ \mathbf{n} $ we adopt the convention $ \langle q \rangle = q,\,\, \jump{q}=q \mathbf{ n },\,\, \langle \mathbf{r} \rangle = \mathbf{r} $ and $ \jump{\mathbf{r}}=\mathbf{r}\cdot\mathbf{ n }$.

The finite dimensional trial spaces where approximate velocity and pressure will be sought are, respectively, the lowest order Raviart-Thomas space and the space of piecewise constants:
\begin{align*}
U_h=\lbrace \bv_h \in U: \bv_h|_{K} = (a+bx,c+by), \, \forall K\in \mathcal{T}_h  \rbrace,\\
W_h=\lbrace w_h \in W: w_h|_{ K}\,\,\text{is a constant}, \,\forall K\in \mathcal{T}_h  \rbrace.
\end{align*}
We introduce a first dual \emph{diamond} grid
(usually employed in non-conforming FVE methods, see \cite{carstensen16})
required for the approximation of the flow equations. The partition is
denoted by $ \cT_h^*$ and its diamond elements $T_M^*$ are quadrilaterals
associated with an interior edge $e_M$ of $\cT_h$ (whose mid-point is $M$). They are formed by joining the
end points of that edge to the barycentre of the triangles sharing the edge.
For a boundary edge, the diamond element coincides with the boundary sub-triangle
obtained by joining the end points of the boundary edge to its barycentre
(see Figure \ref{mfvm}).

\begin{figure}[t]
\begin{center}
\includegraphics[width = 0.775\textwidth]{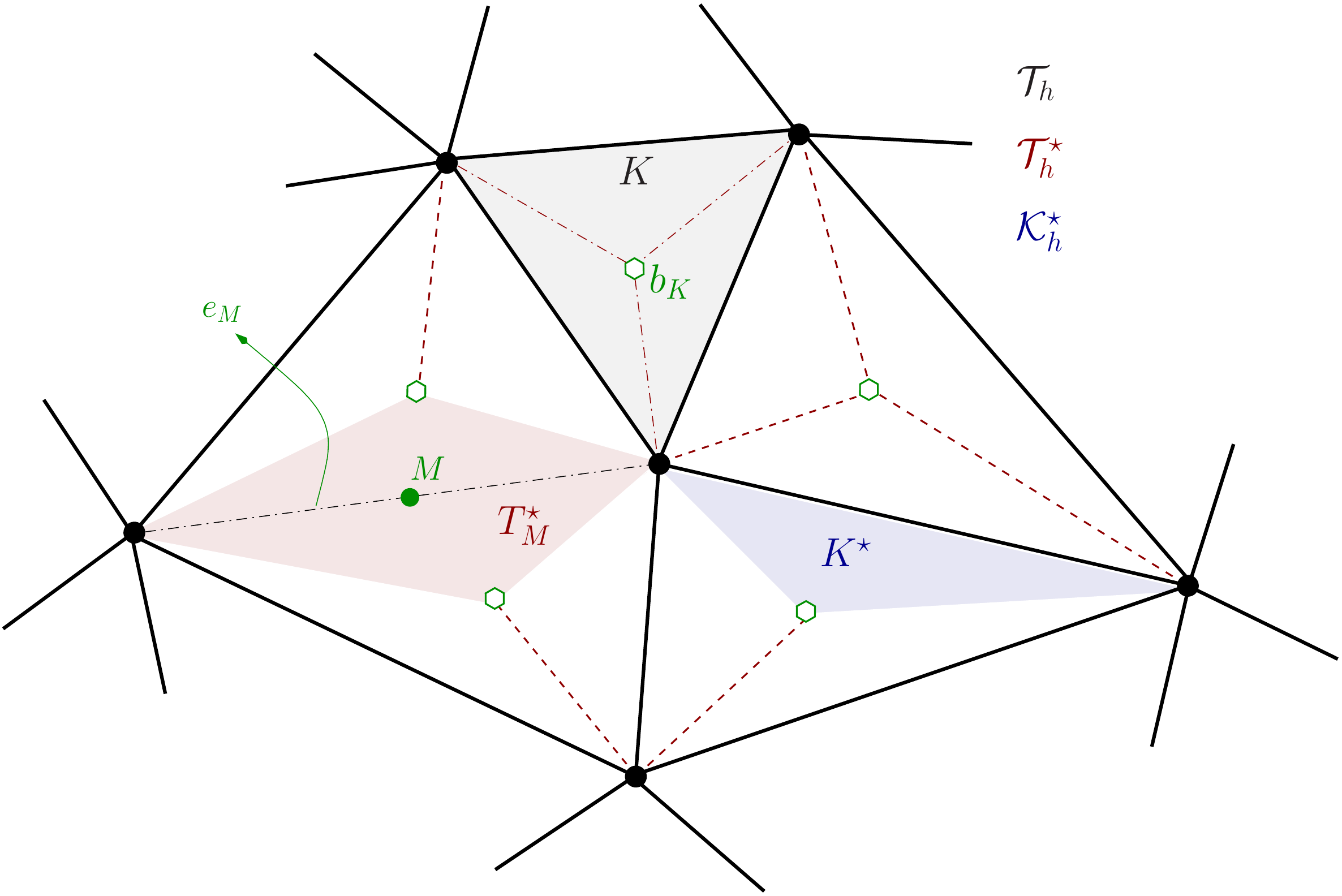}
\end{center}
\caption{Compound of five elements in the primal triangular mesh $\cT_h$ (e.g. $K$ and its barycentre $b_K$),
and examples of one diamond element $T_M^* \in \cT_h^*$ associated to the mid-point $M$ of the
edge $e_M$, and one dual element $K^*\in \mathcal{K}_h^*$.} \label{mfvm}
\end{figure}

The test space for velocity is defined by
\begin{align*}
U_h^*= \lbrace \bv_h \in L^2(\Om)^2: \bv_h|_{T_M^*} \,\, \text{is a constant vector}, \,\, \forall \,T_M^* \in \cT_h^*\,\, \text{and}\,\, \bv_h \cdot \bn=0\,\, \text{on}\,\,\partial \Om  \rbrace.
\end{align*}
The velocity trial and test spaces are connected by a transfer operator $ \gamma_h : U_h \longrightarrow U_h^* $ defined by
\begin{align}
\gamma_h \bv_h(\bx)=\sum\limits_{i=1}^{N_m} \bv_h(M_i) \chi_i^*(\bx) \quad \forall \bx \in \Omega, \label{gm1}
\end{align}
where $M_i$ is the mid-point of a given edge, $ N_m $ is the total number of such mid-side nodes,
and $ \chi_i^*$ is the characteristic function on the diamond $T_{M_i}^*$, that is,
$$\chi_i^*(\bx) = \begin{cases}
    1, & \, \text{if}\,\,\bx\in T_{M_i}^*\\
    0, & \, \text{otherwise}.
\end{cases} $$
The following result collects some properties of $ \gamma_h $, whose proof can be found in \cite{chou98}.
\begin{lemma}\label{prop-gamma}
Let $ \gamma_h $ be the transfer operator defined in \ref{gm1}. Then
\begin{align}
\norm{\gamma_h \bv_h}_{0,\Om} &\leq \norm{\bv_h}_{0,\Om} \quad \forall \bv_h \in U_h, \label{gmhl2}\\
\norm{\bv_h-\gamma_h \bv_h}_{0,\Om} &\leq Ch\norm{\bv_h}_{\div;\Om} \quad \forall \bv_h \in U_h, \label{gmhest}\\
b(\gamma_h \bv_h, w_h) &=-(\nabla\cdot \bv_h, w_h)\quad \forall \bv_h \in U_h, \,\, \forall w_h \in W_h,\label{brel}\\
(\alpha({c}_h){\bv}_h,\gamma_h \bv_h) & \geq C\norm{\bv_h}_{\div;\Om}^2 \quad \forall \bv_h \in U_h \,\, \text{with}\,\, \nabla\cdot \bv_h=0.
\end{align}
\end{lemma}

For a fixed value of the approximate saturation, $ \hat{c}_h$ to be made precise later, let us
consider a fixed control $q$. Then, we can proceed as in \cite{kumar13} and define
an approximation of the state flow equations:
Find $ (\hat{\bu}_h,\hat{p}_h): \bar{J}\longrightarrow U_h \times W_h $ such that for $ t\in J $
\begin{align*}
(\alpha(\hat{c}_h)\hat{\bu}_h,\gamma_h \bv_h)+b(\gamma_h \bv_h,\hat{p}_h) & =0, \quad \forall \bv_h \in U_h,\\
(\nabla \cdot\hat{\bu}_h,w_h) - ((r_0-r_1)q,w_h) &= 0,\quad \forall w_h \in W_h,
\end{align*}
where
\begin{align*}
b(\gamma_h \bv_h, w_h): =-\sum\limits_{i=1}^{N_m} \bv_h(M_i)\cdot \int_{\partial T_{M_i}^*} w_h \bn_{T_{M_i}^*}\ds \quad \forall \bv_h \in U_h, \quad \forall w_h \in W_h.
\end{align*}

In addition to the diamond mesh $ \cT_h^*$ we  introduce a second auxiliary partition
$\mathcal{K}_h^* $, on which the DFVE approximation of the saturation will be carried out.
The elements in $\mathcal{K}_h^*$ are constructed by dividing each primal element
$ K \in \cT_h $ into three sub-triangles by joining the barycentre $ b_K $ with the
 vertices of $ K $. We can then define
the trial space $ M_h $ on $ \cT_h $ and the test space $ L_h $ on $ \mathcal{K}_h^* $ for the
saturation approximation as
\begin{align*}
M_h=\lbrace z_h \in L^2(\Omega): z_h|_ K \in \mathcal{P}_1( K)\quad \forall  K\in \cT_h \rbrace,\\
L_h=\lbrace z_h \in L^2(\Omega): z_h|_K^* \in \mathcal{P}_0(K^*)\quad \forall K^*\in \mathcal{K}_h^* \rbrace,
\end{align*}
where $ \mathcal{P}_k( K) $ denotes the local space of polynomials of degree up to $k$.
We also introduce a discrete space with higher regularity
$ M(h)=M_h\cap H^2(\Omega)$, and (as done for the approximation of velocity) we are able to
map trial and test spaces thanks to the transfer operator $\eta_h: M(h)\to L_h $ defined by
\begin{equation}
\eta_hz|_{K^*}=\frac{1}{h_e} \int_{e} z|_{K^*}\ds,\qquad  K^*\in \mathcal{K}_h^*, \label{gm2}
\end{equation}
with $ h_e $ denoting the length of the edge $e\in \partial K$ which is part of the
dual element $ K^* $ (see Figure \ref{mfvm}). In analogy to Lemma~\ref{prop-gamma}, we
now state some properties of this map, necessary in our subsequent analysis. For a
proof we refer to \cite{bi10,kumar13,ye04}.
\begin{lemma}
For the operator $ \eta_h$ defined in \eqref{gm2}, the following properties hold:
\begin{enumerate}
\item The norm defined by $\nnorm{z_h}_{\eta_h}^2:=(z_h,\eta_hz_h)$, for $ z_h \in M_h $, is equivalent to the $L^2-$norm.
\item The operator $ \eta_h$ is stable with respect to the $L^2-$norm. In particular
\begin{equation}
 \norm{\eta_hz_h}_{0,\Om}=\norm{z_h}_{0,\Om}, \,\, \forall z_h \in M_h. \label{gml2}
\end{equation}
\item There holds
$ \norm{z-\eta_hz}_{0, K} \leq  Ch_{ K}\norm{z}_{1, K} $ for all $ z\in M(h) $ and $  K\in \cT_h $.
\end{enumerate}
\end{lemma}

The DFVE formulation for the saturation equation in the state system for a given  control
$ q $ can be defined as: Find $ \hat{c}_h(t)\in M_h $, $ t\in \bar{J} $ such that
\begin{align*}
(\phi \partial_t \hat{c}_h, \eta_hz_h)+A_h(\hat{c}_h;\hat{c}_h,z_h)+(b(\hat{c}_h)\hat{\bu}_h\cdot \nabla \hat{c}_h,\eta_hz_h) =(f(\hat{c}_h)r_0 q, \eta_hz_h),\quad \forall z_h \in M_h,
\end{align*}
associated with initial condition $ \hat{c}_h(0)=c_{0,h} $, where $ \hat{c}_{0,h} $ is a Riesz projection of $ c_0(\bx) $, and for $ z,\phi,\psi \in M(h) $, the trilinear form $ A_h(\cdot;\cdot,\cdot) $ is defined by
\begin{align}
A_h(\psi;\phi,z)=&-\sum\limits_{ K\in\cT_h}\sum\limits_{j=1}^3\int\limits_{v^{j+1}_Kb_Kv_K^j} \mathcal{D}(\psi)\nabla \phi\cdot \bn \eta_hz \ds -\sum\limits_{e\in \mathcal{E}_h}\int \limits_e \jump{\eta_hz}\cdot \langle \mathcal{D}(\psi)\nabla \phi \rangle \ds
\nonumber\\&-\sum\limits_{e\in \mathcal{E}_h}\int \limits_e \jump{\eta_h\phi}\cdot \langle \mathcal{D}(\psi)\nabla z \rangle \ds+\sum\limits_{e\in \mathcal{E}_h}\int \limits_e \frac{\xi}{h_e}\jump{\phi}\jump{z}\ds,  \label{bil}
\end{align}
where $v_K^j$ denotes a given vertex of the primal element $K\in\cT_h$  and we adopt the convention $ v_K^4=v_K^1 $. The
parameter $\xi$ is a penalisation constant, chosen independently of $ h $. It turns out that the bilinear form defined in \eqref{bil} is bounded and coercive with respect to the mesh dependent norm $\nnormh{\cdot} $ defined by (see \cite[Lemmas 2.3,2.4]{kumar09}):
\begin{align*}
\nnormh{z_h}^2:=\sum\limits_{ K\in\cT_h} |z_h|_{1, K}^2+\sum_{e\in \mathcal{E}_h} \frac{1}{h_e}\int_e \jump{z_h}^2\ds.
\end{align*}

Applying the combined mixed FVE/DFVE schemes for the space discretisation of the optimal control problem \eqref{obj2}-\eqref{sys2} and relation \eqref{brel}, we obtain the following semidiscrete formulation: Find $ ({\bu}_h(t),{p}_h(t),c_h(t),{\bu}_h^*(t),{p}_h^*(t),c_h^*(t),q_h)\in U_h\times W_h\times M_h\times U_h \times W_h \times M_h\times \Qad  $ with $ t\in \bar{J} $ satisfying
\begin{align}
(\alpha(c_h)\bu_h,\gamma_h \bv_h)-( \nabla\cdot \bv_h,p_h) =0, \quad \forall \bv_h \in U_h, \label{sdvp}\\
(\nabla \cdot\bu_h,w_h) =((r_0-r_1)q_h,w_h),\quad \forall w_h \in W_h, \label{sdp}\\
(\phi \partial_t c_h, \eta_hz_h)+A_h(c_h;c_h,z_h)+(b(c_h)\bu_h\cdot\nabla c_h,\eta_hz_h) =(f(c_h)r_0 q_h, \eta_hz_h),\,\, \forall z_h \in M_h, \label{sdc}\\
(\alpha(c_h)\bu_h^*,\gamma_h \bv_h)-(\nabla\cdot \bv_h,p_h^*)+(c_h^*b(c_h)\nabla c_h,\gamma_h \bv_h) =0, \quad \forall \bv_h \in U_h, \label{sdavp}\\
(\nabla \cdot\bu_h^*,w_h) =0,\quad \forall w_h \in W_h, \label{sdap}\\
\left.
 \begin{aligned}
-(\phi\partial_t c_h^*, \eta_hz_h)+A_h(c_h;c_h^*,z_h)-(b(c_h)\bu_h\cdot\nabla c_h^*, \eta_hz_h)+(\mathcal{D}'(c_h)\nabla c_h\cdot \nabla c_h^*, \eta_hz_h)\\+(\alpha'(c_h)\bu_h^*\cdot \bu_h,\eta_hz_h)+(r_1q_hb(c_h)c_h^*,\eta_hz_h) =(wc_h, \eta_hz_h),\,\, \forall z_h \in M_h,
\end{aligned}
 \right\}\label{sdac}\\
\int_0^T (f(c_h)r_0c_h^*-(r_0-r_1)p_h^*+\alpha_0 q_h, \tilde{q}-q_h) \dt \geq 0, \quad \forall \tilde{q} \in \Qad , \label{sdopt}
\end{align}
subject to the initial and final conditions $c_h(0)=c_{0,h}$, $c_h^*(T)=0$.

\paragraph{Temporal discretisation.} Let $ \lbrace t^i \rbrace_{i=0}^N $ be a uniform partition of time interval $ [0,T] $ with time step $ \Delta t>0 $. We apply a backward Euler
method to advance in time the optimal control system \eqref{sdvp}-\eqref{sdopt},
leading to the following fully-discrete formulation:
Find $ ({\bu}_h^i,{p}_h^i,c_h^{i+1},{\bu}_h^{*i},{p}_h^{*i},{c}_h^{*(i+1)},q_h^i)\in U_h\times W_h\times M_h\times U_h \times W_h \times M_h\times \Qad  $ such that
\begin{align*}
(\alpha(c_h^{i})\bu_h^i,\gamma_h \bv_h)-( \nabla\cdot \bv_h,p_h^i)=0, \,\, i=0,\ldots,N;\\
(\nabla \cdot\bu_h^i,w_h) =((r_0-r_1)q_h^i,w_h),\,\, i=0,\ldots,N;\\
(\phi\frac{c_h^{i+1}-c_h^i}{\Delta t}, \eta_hz_h)+A_h(c_h^{i+1};c_h^{i+1},z_h)+(b(c_h^{i+1})\bu_h^i\cdot\nabla c_h^{i+1},\eta_hz_h) \qquad \qquad \\
=(f(c_h^{i+1})r_0 q_h^{i+1}, \eta_hz_h),\,\,i=0,\ldots,N-1;\\
(\alpha(c_h^{i})\bu_h^{*i},\gamma_h \bv_h)-(\nabla\cdot \bv_h,p_h^{*i})+(c_h^{*i}b(c_h^i)\nabla c_h^i,\gamma_h \bv_h) =0,\,\,i=N,\ldots,0; \\
(\nabla \cdot\bu_h^{*i},w_h) =0,\,\,i=N,\ldots,0;\\
-(\phi\frac{c_h^{*(i+1)}-c_h^{*i}}{\Delta t}, \eta_hz_h)+A_h(c_h^{i+1};c_h^{*(i+1)},z_h)-(b(c_h^{i+1})\bu_h^i\cdot\nabla c_h^{*(i+1)}, \eta_hz_h) \qquad \qquad \\
+(\mathcal{D}'(c_h^{i+1})\nabla c_h^{i+1}\cdot \nabla c_h^{*(i+1)}, \eta_hz_h)+(\alpha'(c_h^{i+1})\bu_h^{*i}\cdot \bu_h^i,\eta_hz_h)\qquad \qquad \\
+(r_1q_h^{i+1}b(c_h^{i+1})c_h^{*(i+1)},\eta_hz_h)-(wc_h^{i+1}, \eta_hz_h) =0,\,\, i=N-1,\ldots,0;\\
(f(c_h^i)r_0c_h^{*i}-(r_0-r_1)p_h^{*i}+\alpha_0 q_h^i, \tilde{q_h}-q_h^i) \geq 0, \quad \forall \tilde{q_h} \in \Qad ,\,\,i=0,\ldots,N;
\end{align*}
for all $ \bv_h \in U_h,\, w_h \in W_h  $ and $ z_h \in M_h $,  with initial and terminal conditions $ c_h^0=c_{0,h}, c_h^{*T}=0 $.

\section{Error estimates}\label{sec:error}

In this section, we derive suitable error bounds for the mixed FVE and DFVE approximations of \eqref{obj2}-\eqref{sys2} for a fixed local reference control satisfying the optimality conditions \eqref{ncd} and \eqref{suf}. Our analysis requires similar assumptions as those adopted
in \cite[Assumption (C)]{chang14}. More precisely, there exists $ M_2>0 $ such that:
\begin{align*}
\norm{\alpha''(c)}_{L^\infty(J;L^\infty)}+\norm{b''(c)}_{L^\infty(J;L^\infty)}+\norm{\mathcal{D}''(c)}_{L^\infty(J;L^\infty)} +
\norm{\bu}_{L^\infty(J;L^2(\Omega)^2)}+\norm{\partial_t \bu}_{L^\infty(J;L^2(\Omega)^2)} \\
+\norm{p}_{L^\infty(J;H^1(\Omega))}
\norm{c}_{L^\infty(J;H^2(\Omega))}+\norm{\partial_t c}_{L^\infty(J;H^2(\Omega))} +
\norm{\bu^*}_{L^\infty(J;L^2(\Omega)^2)}+\norm{\partial_t \bu^*}_{L^\infty(J;L^2(\Omega)^2)} \\
+\norm{p^*}_{L^\infty(J;H^1(\Omega))}
\norm{c^*}_{L^\infty(J;H^2(\Omega))}+\norm{\partial_t c^*}_{L^\infty(J;H^2(\Omega))} \leq M_2.\\
\end{align*}

At each time interval $ [t^m,t^{m+1}], \, m=1,\ldots,N-1 $ and for a given arbitrary $ q^m $, let the functions
$ (\hat{\bu}_h^m,\hat{p}_h^m,\hat{c}_h^{m+1},\hat{\bu}_h^{*m},\hat{p}_h^{*m},\hat{c}_h^{*(m+1)}) $ satisfy the following intermediate system
\begin{align}
\left.
 \begin{aligned}
(\alpha(\hat{c}_h^m)\hat{\bu}_h^m,\gamma_h \bv_h)-(\nabla \cdot \bv_h,\hat{p}_h^m) & =0, \quad \forall \bv_h \in U_h,\\
(\nabla \cdot\hat{\bu}_h^m,w_h) -((r_0-r_1)q^m,w_h) & =0,\quad \forall w_h \in W_h,
\end{aligned}
 \right\} \label{ifu1}\\
 \left.
 \begin{aligned}
(\phi\frac{\hat{c}_h^{m+1}-\hat{c}_h^{m}}{\Delta t}, \eta_hz_h)+A_h(\hat{c}_h^{m+1};\hat{c}_h^{m+1},z_h)+(b(\hat{c}_h^{m+1})\hat{\bu}_h^m\cdot \nabla \hat{c}_h^{m+1},\eta_hz_h) \\=(f(\hat{c}_h^{m+1})r_0 q^{m+1}, \eta_hz_h),\quad \forall z_h \in M_h,
\end{aligned}
 \right\} \label{ifc1}\\
\left.
 \begin{aligned}
(\alpha(\hat{c}_h^{m})\hat{\bu}_h^{*m},\gamma_h \bv_h)-(\nabla\cdot\bv_h,\hat{p}_h^{*m})+(\hat{c}_h^{*m}b(\hat{c}_h^{m})\nabla \hat{c}_h^{m},\gamma_h \bv_h) &=0, \quad \forall \bv_h \in U_h,\\
(\nabla \cdot\hat{\bu}_h^{*m},w_h) &=0,\quad \forall w_h \in W_h,
\end{aligned}
 \right\}\label{ifu2}\\
 \left.
 \begin{aligned}
-(\phi\frac{\hat{c}_h^{*(m+1)}-\hat{c}_h^{*m}}{\Delta t}, \eta_hz_h)+A_h(\hat{c}_h^{m+1};\hat{c}_h^{*(m+1)},z_h)-(b(\hat{c}_h^{m+1})\hat{\bu}_h^m\cdot \nabla \hat{c}_h^{*(m+1)},\eta_hz_h)\\+(\mathcal{D}'(\hat{c}_h^{m+1})\nabla \hat{c}_h^{m+1}\cdot \nabla \hat{c}_h^{*(m+1)}, \eta_hz_h)+(\alpha'(\hat{c}_h^{m+1})\hat{\bu}_h^{*m}\cdot \hat{\bu}_h^{m},\eta_hz_h)\\+(r_1b(\hat{c}_h^{m+1})q^{m+1}\hat{c}_h^{*(m+1)},\eta_hz_h) =(w\hat{c}_h^{m+1}, \eta_hz_h),\quad \forall z_h \in M_h,
\end{aligned}
 \right\} \label{ifc2}
\end{align}
associated with initial and terminal conditions $ \hat{c}_h(0)=c_{0,h}, \quad \hat{c}_h^*(T)=0 $.

The following theorem (whose proof can be found in \cite{kumar12,kumar08}) gives
an error estimate for the intermediate state variables.
\begin{theorem}\label{thm1}
At $ t=t^m,\, 1\leq m \leq N $ and for a given $ q^m $, let $ (\bu^m,p^m,c^m) $ be the exact solutions and $ (\hat{\bu}_h^m,\hat{p}_h^m,\hat{c}_h^m) $ be the solutions of the intermediate system. Then
\begin{equation*}
\norm{\bu^m-\hat{\bu}_h^m}_{0,\Om}+\norm{p^m-\hat{p}_h^m}_{0,\Om}+ \norm{c^m-\hat{c}_h^m}_{0,\Om} \leq C(h+\Delta t).
\end{equation*}
\end{theorem}

Likewise, one can derive a similar error bound for the intermediate costate variables.
\begin{theorem} \label{thm2}
At $ t=t^m,\, 1\leq m \leq N $ and for a given $ q^m $, let $ (\bu^{*m},p^{*m},c^{*m}) $ be the exact solutions and $ (\hat{\bu}_h^{*m},\hat{p}_h^{*m},\hat{c}_h^{*m}) $ be the solutions of the intermediate system. Then
\begin{equation*}
\norm{\bu^{*m}-\hat{\bu}_h^{*m}}_{0,\Om}+\norm{p^{*m}-\hat{p}_h^{*m}}_{0,\Om} + \norm{c^{*m}-\hat{c}_h^{*m}}_{0,\Om} \leq C(h+\Delta t).
\end{equation*}
\end{theorem}
\begin{proof}
At $ t=t^m $ let the auxiliary functions  $ (\tilde{\bu}_h^{*m},\tilde{p}_h^{*m}) $ satisfy the following equations
\begin{equation}
 \begin{split}
(\alpha(c^m)\tilde{\bu}_h^{*m},\bv_h)-(\nabla \cdot\bv_h,\tilde{p}_h^{*m}) & =-(c^{*m}b(c^m)\nabla c^m, \bv_h), \quad \forall \bv_h \in U_h,\\
(\nabla \cdot\tilde{\bu}_h^{*m},w_h) & =0,\quad \forall w_h \in W_h.
\end{split}\label{af2}
\end{equation}
Then, using the Raviart-Thomas and $ L^2-$projections (cf. \cite{brezzi91,chou98}) we can
assert that
\begin{equation}
\norm{\bu^{*m}-\tilde{\bu}_h^{*m}}_{0,\Om}+\norm{p^{*m}-\tilde{p}_h^{*m}}_{0,\Om} \leq Ch\left( \norm{\bu^{*m}}_{1,\Om}+\norm{p^{*m}}_{1,\Om} \right). \label{uptldest}
\end{equation}
Now, we split $ \bu^{*m}-\hat{\bu}_h^{*m}=(\bu^{*m}-\tilde{\bu}_h^{*m})+(\tilde{\bu}_h^{*m}-\hat{\bu}_h^{*m}) $ and $ p^{*m}-\hat{p}_h^{*m}=(p^{*m}-\tilde{p}_h^{*m})+(\tilde{p}_h^{*m}-\hat{p}_h^{*m}) $. Since the estimates of $ \bu^{*m}-\tilde{\bu}_h^{*m} $ and $ p^{*m}-\tilde{p}_h^{*m} $ are known from \eqref{uptldest}, it then suffices to estimate $ \tilde{\bu}_h^{*m}-\hat{\bu}_h^{*m} $ and $ \tilde{p}_h^{*m}-\hat{p}_h^{*m} $. Let $ \tilde{\be}_{1h}^{*m}= \tilde{\bu}_h^{*m}-\hat{\bu}_h^{*m} $ and $ \tilde{e}_{2h}^{*m}= \tilde{p}_h^{*m}-\hat{p}_h^{*m} $.
Subtracting \eqref{ifu2} from \eqref{af2} we have
\begin{align}
(&\alpha(\hat{c}_h^m)\tilde{\be}_{1h}^{*m},\gamma_h \bv_h)-(\nabla\cdot \bv_h,\tilde{e}_{2h}^{*m})= [(\alpha(c^m)\tilde{\bu}_h^{*m},\gamma_h \bv_h-\bv_h)+((\alpha(\hat{c}_h^m)-\alpha(c^m))\tilde{\bu}_h^{*m},\gamma_h \bv_h)]\nonumber\\
&+[(c^{*m}b(c^m)\nabla c^m,\gamma_h \bv_h-\bv_h)+(\hat{c}_h^{*m}b(\hat{c}_h^m)\nabla \hat{c}_h^m-c^{*m}b(c^m)\nabla c^m,\gamma_h \bv_h)], \, \forall \bv_h \in U_h, \label{sub1}
\end{align}
\begin{align}
\text{and}\quad (\nabla \cdot \tilde{\be}_{1h}^{*m},w_h) =0, \quad \forall w_h \in W_h. \label{sub2}
\end{align}
Since $ \nabla\cdot U_h \subset W_h $, we take $ w_h=\nabla\cdot\tilde{\be}_{1h}^{*m}  $ in \eqref{sub2} to obtain $ \norm{\nabla\cdot\tilde{\be}_{1h}^{*m}}_{0,\Om}=0 $,
which further implies (from the definition of $ \norm{\cdot}_{\div,\Om} $) that
\begin{equation}
\norm{\tilde{\be}_{1h}^{*m}}_{\div,\Om}=\norm{\tilde{\be}_{1h}^{*m}}_{0,\Om}. \label{h1l2eq}
\end{equation}
Choosing $ \bv_h=\tilde{\be}_{1h}^{*m} $ in \eqref{sub1} and $ w_h=\tilde{e}_{2h}^{*m} $ in \eqref{sub2}, we arrive at
\begin{align}
C\norm{\tilde{\be}_{1h}^{*m}}_{\div,\Om}^2 & \leq R_1 + R_2 : =
\bigl[(\alpha(c)\tilde{\bu}_h^{*m},\gamma_h \tilde{\be}_{1h}^{*m}-\tilde{\be}_{1h}^{*m})+((\alpha(\hat{c}_h^m)-\alpha(c^m))\tilde{\bu}_h^{*m},\gamma_h \tilde{\be}_{1h}^{*m})\bigr] \nonumber\\
&\ +\bigl[(c^{*m}b(c^m)\nabla c^m,\gamma_h \tilde{\be}_{1h}^{*m}-\tilde{\be}_{1h}^{*m})+(\hat{c}_h^{*m}b(\hat{c}_h^m)\nabla \hat{c}_h^m-c^{*m}b(c^m)\nabla c^m,\gamma_h \tilde{\be}_{1h}^{*m})\bigr]. \label{e1}
\end{align}
Using then \eqref{gmhest}, the Lipschitz continuity of $ \alpha $, and \eqref{gmhl2}, the
first term in \eqref{e1} can be bounded as
\begin{align*}
R_1 & \leq C \left(\norm{\tilde{\bu}_h^{*m}}_{0,\Om}\norm{\tilde{\be}_{1h}^{*m}-\gamma_h \tilde{\be}_{1h}^{*m}}_{0,\Om}+\norm{c^m-\hat{c}_h^m}_{0,\Om}\norm{\tilde{\bu}_h^*}_{L^\infty(\Om)^2}\norm{\gamma_h\tilde{\be}_{1h}^*}_{0,\Om}\right)\\
& \leq C \left(h\norm{\tilde{\bu}_h^*}_{0,\Om}\norm{\tilde{\be}_{1h}^*}_{\div,\Om}+\norm{c-\hat{c}_h}_{0,\Om}\norm{\tilde{\bu}_h^*}_{L^\infty(\Om)^2}\norm{\tilde{\be}_{1h}^*}_{0,\Om}\right).
\end{align*}
Regarding the second term in \eqref{e1}, we use \eqref{gmhl2} and \eqref{gmhest} to obtain
\begin{align*}
R_2 \leq &\, C(h\norm{c^{*m}}_{L^\infty(\Om)}\norm{\nabla c^m}_{L^\infty(\Om)}\norm{\tilde{\be}_{1h}^{*m}}_{\div,\Om}+\norm{c^m-\hat{c}_h^m}_{0,\Om}\norm{\nabla \hat{c}_h^{*m}}_{L^\infty(\Om)}\norm{\tilde{\be}_{1h}^{*m}}_{0,\Om}\\&+\norm{c^{*m}-\hat{c}_h^{*m}}_{0,\Om}\norm{\nabla c^m}_{L^\infty(\Om)}\norm{\tilde{\be}_{1h}^{*m}}_{0,\Om}).
\end{align*}
Substituting these bounds back in \eqref{e1}, and using \eqref{h1l2eq}, we arrive at
$$
\norm{\tilde{\be}_{1h}^{*m}}_{0,\Om}\leq C\biggl(\norm{c^m-\hat{c}_h^m}_{0,\Om}+\norm{c^{*m}-\hat{c}_h^{*m}}_{0,\Om}\biggr).$$
Next, to estimate $ \norm{\tilde{e}_{2h}^{*m}} $ we can choose
$ \bv_h=\tilde{\be}_{1h}^{*m} $ in \eqref{sub1}, leading to
$$
(\nabla\cdot \tilde{\be}_{1h}^{*m},\tilde{e}_{2h}^{*m}) \leq C\left[\norm{c^m-\hat{c}_h^m}_{0,\Om}+\norm{c^{*m}-\hat{c}_h^{*m}}_{0,\Om}+\norm{\tilde{\be}_{1h}^{*m}}_{0,\Om}\right]\norm{\tilde{\be}_{1h}^{*m}}_{0,\Om},
$$
which, after applying the inf-sup condition, gives
$$
\norm{\tilde{e}_{2h}^{*m}}_{0,\Om} \leq C\left[\norm{c^m-\hat{c}_h^m}_{0,\Om}+\norm{c^{*m}-\hat{c}_h^{*m}}_{0,\Om}+\norm{\tilde{\be}_{1h}^{*m}}_{0,\Om}\right],
$$
and so we have
\begin{equation}
\norm{\bu^{*m}-\hat{\bu}_h^{*m}}_{L^2(\Om)^2}+\norm{p^{*m}-\hat{p}_h^{*m}}_{0,\Om}\leq C[\norm{c^m-\hat{c}_h^m}_{0,\Om}+\norm{c^{*m}-\hat{c}_h^{*m}}_{0,\Om}]. \label{uspsrelcs}
\end{equation}

Now, for a fixed $t=t^n$, let $ \tilde{c}_h^{*n} $ denote the Riesz projection of $c^{*n}$. We then have that for any $ z_h \in M_h $, the following condition holds
\begin{equation}
A_h(c^n;c^{*n}-\tilde{c}_h^{*n},z_h)-((b(c^n)\bu^n-\mathcal{D}'(c^n)\nabla c^n)\cdot \nabla(c^{*n}-\tilde{c}_h^{*n}),z_h)+\lambda(c^{*n}-\tilde{c}_h^{*n},z_h)=0, \label{ogl}
\end{equation}
where $ \lambda >0$ is chosen such that, if fixing the first argument of
the trilinear form in \eqref{ogl}, the resulting bilinear form is coercive with respect to the norm $ \nnormh{\cdot} $.
We then write $ c^{*n}-\hat{c}_h^{*n}=(c^{*n}-\tilde{c}_h^{*n})+(\tilde{c}_h^{*n}-\hat{c}_h^{*n})=\rho^{*n}+\theta^{*n} $. Since the estimates for $ \rho^{*n} $ are known (see \cite{kumar12,kumar08}), it only remains to derive bounds for $ \theta^{*n} $. We proceed to multiply the costate saturation equation \eqref{adjc} by $ \eta_hz_h $, and integrating over $ \Omega $ we have (at $ t=t^{n+1} $)
\begin{equation}\label{ifcc2}
 \begin{split}
-(\phi\partial_t c^{*(n+1)}, \eta_hz_h)-((b(c^{n+1})\bu^{n+1}-\mathcal{D}'(c^{n+1})\nabla c^{n+1})\cdot \nabla c^{*(n+1)},\eta_hz_h)+A_h(c^{n+1};c^{*(n+1)},z_h) \\
+(\alpha'(c^{n+1}){\bu}^{*(n+1)}\cdot{\bu^{n+1}},\eta_hz_h)+(r_1q^{n+1}b(c^{n+1}) c^{*(n+1)},\eta_hz_h)=(wc^{n+1}, \eta_hz_h).
\end{split}
\end{equation}
Subtracting the intermediate discrete costate equation \eqref{ifc2} from \eqref{ifcc2} yields
\begin{align*}
&\qquad-(\phi\frac{\theta^{*(n+1)}-\theta^{*n}}{\Delta t}, \eta_hz_h)+A_h(c^{n+1};c^{*(n+1)},z_h)-A_h(\hat{c}_h^{n+1};\hat{c}_h^{*(n+1)},z_h)\\
&\qquad -((b(c^{n+1})\bu^{n+1}-\mathcal{D}'(c^{n+1})\nabla c^{n+1})\cdot \nabla c^{*(n+1)},\eta_hz_h)\\
&\qquad +(r_1q^{n+1} \theta^{*(n+1)},\eta_hz_h)+((b(\hat{c}_h^{n+1})\hat{\bu}_h^n-\mathcal{D}'(\hat{c}_h^{n+1})\nabla \hat{c}_h^{n+1})\cdot \nabla \hat{c}_h^{*(n+1)},\eta_hz_h)\\
&=\ (\phi\frac{\rho^{*(n+1)}-\rho^{*n}}{\Delta t}, \eta_hz_h)+\phi(\partial_t c^{*(n+1)}-\frac{c^{*(n+1)}-c^{*n}}{\Delta t}, \eta_hz_h)\\
& \qquad -(r_1q^{n+1}\rho^{*(n+1)},\eta_hz_h)-(r_1q^{n+1}c_h^{*(n+1)}(b(c^{n+1})-b(c_h^{n+1})),\eta_hz_h)\\
& \qquad +(w(c^{n+1}-\hat{c}_h^{n+1}), \eta_hz_h)+(\alpha'(\hat{c}_h^{n+1})\hat{\bu}_h^{*n}\cdot \hat{\bu}_h^n-\alpha'(c^{n+1}){\bu}^{*(n+1)}\cdot{\bu^{n+1}},\eta_hz_h).
\end{align*}
Utilising relation \eqref{ogl} and choosing $ z_h=\theta^{*(n+1)} $ in the previous
equation, we can write
\begin{equation}\label{ceq2}
 \begin{split}
&\quad -(\phi\frac{\theta^{*(n+1)}-\theta^{*n}}{\Delta t}, \eta_h\theta^{*(n+1)})+A_h(\hat{c}_h^{n+1};\theta^{*(n+1)},\theta^{*(n+1)})
-((b(c^{n+1})\bu^{n+1}\\
& \quad -\mathcal{D}'(c^{n+1})\nabla c^{n+1})\cdot \nabla \theta^{*(n+1)},\eta_h\theta^{*(n+1)})+(r_1q^{n+1} \theta^{*(n+1)},\eta_h\theta^{*(n+1)})\\
&=  (\phi\frac{\rho^{*(n+1)}-\rho^{*n}}{\Delta t}, \eta_h\theta^{*(n+1)})+\phi(\partial_t c^{*(n+1)}-\frac{c^{*(n+1)}-c^{*n}}{\Delta t}, \eta_h\theta^{*(n+1)})\\
& \quad -(r_1q^{n+1}\rho^{*(n+1)},\eta_h\theta^{*(n+1)})-(\lambda \rho^{*(n+1)},\eta_h \theta^{*(n+1)})+(w(c^{n+1}-\hat{c}_h^{n+1}), \eta_h\theta^{*(n+1)})\\
& \quad -(r_1q^{n+1}c_h^{*(n+1)}(b(c^{n+1})-b(c_h^{n+1})),\eta_h\theta^{*(n+1)})+A_h(\hat{c}_h^{n+1};\tilde{c}_h^{*(n+1)},\theta^{*(n+1)})\\
& \quad +(\alpha'(\hat{c}_h^{n+1})\hat{\bu}_h^{*n}\cdot \hat{\bu}_h^n-\alpha'(c^{n+1}){\bu}^{*(n+1)}\cdot{\bu^{n+1}},\eta_h\theta^{*(n+1)})-A_h(c^{n+1};\tilde{c}_h^{*(n+1)},\theta^{*(n+1)})\\
& \quad +((b(\hat{c}_h^{n+1})\hat{\bu}_h^n-\mathcal{D}'(\hat{c}_h^{n+1})\nabla \hat{c}_h^{n+1})\cdot \nabla \hat{c}_h^{*(n+1)},\theta^{*(n+1)}-\eta_h\theta^{*(n+1)})\\
& \quad -((b(c^{n+1})\bu^{n+1}-\mathcal{D}'(c^{n+1})\nabla c^{n+1})\cdot \nabla c^{*(n+1)},\theta^{*(n+1)}-\eta_h\theta^{*(n+1)})\\
& \quad +((b(c^{n+1})\bu^{n+1}-b(\hat{c}_h^{n+1})\hat{\bu}_h^n+\mathcal{D}'(\hat{c}_h^{n+1})\nabla \hat{c}_h^{n+1}-\mathcal{D}'(c^{n+1})\nabla c^{n+1})\cdot \nabla \tilde{c}_h^{*(n+1)},\theta^{*(n+1)}).
\end{split}
\end{equation}
Then, thanks to Cauchy-Schwarz inequality and \eqref{gml2}, we can deduce that
$$
(\phi\frac{\rho^{*(n+1)}-\rho^{*n}}{\Delta t}, \eta_h\theta^{*(n+1)})\leq C(\Delta t)^{-1/2}\norm{\partial_t \rho^*}_{L^2(t_n,t_{n+1}; L^2(\Omega))}\|\theta^{*(n+1)}\|_{0,\Om},
$$
and expanding in Taylor series it follows that
$$
(\phi\partial_t c^{*(n+1)}-\phi\frac{c^{*(n+1)}-c^{*n}}{\Delta t}, \eta_h\theta^{*(n+1)})\leq C \left(\Delta t \int_{t_n}^{t_{n+1}}\norm{\partial_{tt}c^*}_{0,\Om}^2ds \right)^{1/2}\|\theta^{*(n+1)}\|_{0,\Om}.
$$

Next, exploiting similar arguments as in the proof of \cite[Lemma 5.3]{chang14}, we can bound the terms in \eqref{ceq2} and apply Young's inequality to obtain
\begin{equation}\label{csbd}
\begin{split}
&-(\phi\frac{\theta^{*(n+1)}-\theta^{*n}}{\Delta t}, \eta_h\theta^{*(n+1)})+A_h(\hat{c}_h^{n+1};\theta^{*(n+1)},\theta^{*(n+1)})  \\
& \quad \leq C[\norm{c^{n+1}-\hat{c}_h^{n+1}}_{0,\Om}^2+\norm{\bu^{*n}-\hat{\bu}^{*n}_h}_{0,\Om}^2+\norm{\bu^n-\hat{\bu}_h^n}_{0,\Om}^2 +\Delta t\norm{\partial_{tt}c^*}_{L^2(t_n,t_{n+1}; L^2(\Omega))}^2\\
& \qquad +\Delta t \norm{\bu_t}_{L^2(t_n,t_{n+1}; L^2(\Omega)^2)}^2+\|\rho^{*(n+1)}\|_{0,\Om}^2+(\Delta t)^{-1}\norm{\partial_t \rho^*}_{L^2(t_n,t_{n+1}; L^2(\Omega))}^2 \! +\|\theta^{*(n+1)}\|_{0,\Om}^2].
\end{split} \end{equation}
On the other hand, noting that $ (\cdot,\eta_h\cdot)\geq 0 $ allows us to
write
\begin{equation}
-(\phi \frac{\theta^{*(n+1)}-\theta^{*n}}{\Delta t}, \eta_h\theta^{*(n+1)})\geq \frac{\phi}{2\Delta t}\left[ (\theta^{*n},\eta_h\theta^{*n})-(\theta^{*(n+1)},\eta_h\theta^{*(n+1)})\right]. \label{tsieq}
\end{equation}
Then, from \eqref{tsieq} together with the coercivity of $ A_h $ and the definition of
$ \nnorm{\cdot}_{\eta_h} $ in \eqref{csbd}, we can sum over $ n=m,\ldots,N-1$ to obtain
\begin{align*}
\nnorm{\theta^{*m}}_{\eta_h}^2  \leq & C \Delta t \sum_{n=m}^{N-1} [ \norm{c^{n+1}-\hat{c}_h^{n+1}}_{0,\Om}^2+\norm{\bu^{*n}-\hat{\bu}^{*n}_h}_{0,\Om}^2+\norm{\bu^n-\hat{\bu}_h^n}_{0,\Om}^2+\Delta t\norm{\partial_{tt}c^*}_{L^2(0,T; L^2(\Omega))}^2\\
& \qquad +\Delta t\norm{\bu_t}_{L^2(0,T; L^2(\Omega)^2)}^2+\|\rho^{*(n+1)}\|_{0,\Om}^2+(\Delta t)^{-1}\norm{\partial_t \rho^*}_{L^2(0,T; L^2(\Omega))}^2+\|\theta^{*(n+1)}\|_{0,\Om}^2].
\end{align*}
Finally, we combine the discrete Gronwall's lemma, the equivalence of the norms $ \nnorm{\cdot}_{\eta_h} $ and $ \norm{\cdot}_{0,\Om} $, Theorem \ref{thm1}, relation \eqref{uspsrelcs}, and the available
estimates for $ \rho^* $, to obtain the bound $\norm{\theta^{*m}}_{0,\Om}\leq C(h+\Delta t) $, which in turn implies that
\begin{equation}
\norm{c^{*m}-\hat{c}_h^{*m}}_{0,\Om}\leq C(h+\Delta t). \label{csest}
\end{equation}
Putting together \eqref{csest} with the result from Theorem \ref{thm1} in \eqref{uspsrelcs}, we
can also derive the estimate
$$\norm{\bu^{*m}-\hat{\bu}_h^{*m}}_{0,\Om}+\norm{p^{*m}-\hat{p}_h^{*m}}_{0,\Om} \leq C(h+\Delta t). $$
\end{proof}

In what follows, for a given time $t^m$ we will adopt the notation
$$(\bu^m(q_h),p^m(q_h),c^m(q_h),\bu^{*m}(q_h),p^{*m}(q_h),c^{*m}(q_h)),$$
to indicate functions satisfying the continuous optimal system for a given control $q_h$.
\begin{theorem} \label{thm3}
For a fixed $ t=t^m, \, 1\leq m \leq N $, let $ q^m $ be a local optimal control of \eqref{obj2}-\eqref{sys2} having  state and costate solutions $(\bu^m,p^m,c^m,\bu^{*m},p^{*m},c^{*m})$, and let $ (q_h^m,\bu_h^m,p_h^m,c_h^m,\bu_h^{*m},p_h^{*m},c_h^{*m})$ be its discrete counterpart. Then, there
exists $C>0$ independent of $h,\Delta t$, such that:
\begin{align*}
\norm{q^m-q_h^m}_{L^2(0,T)} \leq C(h+\Delta t),\\
\norm{\bu^m-{\bu}_h^m}_{0,\Om}+\norm{p^m-{p}_h^m}_{0,\Om}+\norm{c^m-{c}_h^m}_{0,\Om} \leq C(h+\Delta t),\\
\norm{\bu^{*m}-{\bu}_h^{*m}}_{0,\Om}+\norm{p^{*m}-{p}_h^{*m}}_{0,\Om}+\norm{c^{*m}-{c}_h^{*m}}_{0,\Om} \leq C(h+\Delta t).
\end{align*}
\end{theorem}
\begin{proof}
The continuous and discrete variational inequalities readily imply that
\begin{align}
& (f(c^{m})r_0c^{*m}-(r_0-r_1)p^{*m}+\alpha_0 q^m, q^m-q_h^m) \nonumber\\
& \qquad \qquad \leq 0 \leq (f(c_h^{m})r_0c_h^{*m}-(r_0-r_1)p_h^{*m}+\alpha_0 q_h^m, q^m-q_h^m). \label{vieqrel}
\end{align}
On the other hand, taking $\tilde{q}=q^m-q_h^m $, and using the convexity assumption \eqref{suf}, leads to
\begin{align*}
C_0 \norm{q^m-q_h^m}_{L^2(0,T)}^2 &  \leq  (J'(q^m)-J'(q_h^m),q^m-q_h^m),\\
& \leq   (f(c^{m+1})r_0c^{*m}-(r_0-r_1)p^{*m}+\alpha_0 q^m, q^m-q_h^m)\\&-(f(c^{m}(q_h))r_0c^{*m}(q_h)-(r_0-r_1)p^{*m}(q_h)+\alpha_0 q_h^m, q^m-q_h^m),
\end{align*}
and from \eqref{vieqrel}, we have
\begin{align*}
C_0 \norm{q^m-q_h^m}_{L^2(0,T)}^2 & \leq \ (f(c_h^{m})r_0c_h^{*m}-(r_0-r_1)p_h^{*m}+\alpha_0 q_h^m, q^m-q_h^m) \\
& \qquad -(f(c^{m}(q_h))r_0c^{*m}(q_h)-(r_0-r_1)p^{*m}(q_h)+\alpha_0 q_h^m, q^m-q_h^m)\\
& = \ (r_0(f(c_h^{m})c_h^{*m}-f(c^{m}(q_h))c^{*m}(q_h)), q^m-q_h^m)\\
&\qquad -((r_0-r_1)(p_h^{*m}-p^{*m}(q_h), q^m-q_h^m),
\end{align*}
which in turn yields
\begin{equation}
\norm{q^m-q_h^m}_{L^2(0,T)} \leq C\biggl(\norm{c^{m}(q_h)-c_h^{m}}_{0,\Om}+\norm{c^{*m}(q_h)-c_h^{*m}}_{0,\Om}+\norm{p^{*m}(q_h)-p_h^{*m}}_{0,\Om}\biggr). \label{qbd}
\end{equation}
From these results, and proceeding very much in the same way as done in the proofs of
Theorems \ref{thm1} and \ref{thm2}, we can assert  that
\begin{align}
\norm{c^{m}(q_h)-c_h^{m}}_{0,\Om}+\norm{\bu^{m}(q_h)-\bu_h^{m}}_{0,\Om}+\norm{p^{m}(q_h)-p_h^{m}}_{0,\Om}\leq C(h+\Delta t),\label{cqh}\\
\norm{c^{*m}(q_h)-c_h^{*m}}_{0,\Om}+\norm{\bu^{*m}(q_h)-\bu_h^{*m}}_{0,\Om}+\norm{p^{*m}(q_h)-p_h^{*m}}_{0,\Om} \leq C(h+\Delta t),\\
\norm{c^{m}-c_h^{m}}_{0,\Om}+\norm{\bu^{m}-\bu_h^{m}}_{0,\Om}+\norm{p^{m}-p_h^{m}}_{0,\Om} \leq C[(h+\Delta t)+\norm{q^m-q_h^m}_{L^2(0,T)}],\\
\norm{c^{*m}-c_h^{*m}}_{0,\Om}+\norm{\bu^{*m}-\bu_h^{*m}}_{0,\Om}+\norm{p^{*m}-p_h^{*m}}_{0,\Om} \leq C[(h+\Delta t)+\norm{q^m-q_h^m}_{L^2(0,T)}], \label{cstqh}
\end{align}
and hence the desired result follows directly from \eqref{qbd} and \eqref{cqh}-\eqref{cstqh}.
\end{proof}

Next we devote ourselves to the derivation of error estimates for the saturation in the broken $H^1-$norm. Let us start by
introducing the trilinear form $\tilde{A}_h(\cdot;\cdot,\cdot):M(h)^3 \to \mathbb{R}$ defined as
\begin{align*}
\tilde{A}_h(\psi;\phi,z)=&-\sum\limits_{K\in\cT_h}\int\limits_{K} \mathcal{D}(\psi)\nabla \phi\cdot \nabla z \ds -\sum\limits_{e\in \mathcal{E}_h}\int \limits_e \jump{z}\cdot \langle \mathcal{D}(\psi)\nabla \phi \rangle \ds
\nonumber\\&-\sum\limits_{e\in \mathcal{E}_h}\int \limits_e \jump{\phi}\cdot \langle \mathcal{D}(\psi)\nabla z \rangle \ds+\sum\limits_{e\in \mathcal{E}_h}\int \limits_e \frac{\xi}{h_e}\jump{\phi}\jump{z}\ds.
\end{align*}
If we now fix $\psi$ and
set $ \epsilon_a(\psi,\phi,\chi):=\tilde{A}_h(\psi;\phi,\chi)-A_h(\psi;\phi,\chi) \quad \forall \psi,\chi \in M_h $, then we have
the following bound (see \cite[Lemma 3.2]{bi12})
\begin{align}
\epsilon_a(\psi,\phi,\chi) \leq Ch\nnormh{\phi}\nnormh{\chi}. \label{epsrel}
\end{align}

\begin{theorem}
At $ t=t^m, \, 1\leq m \leq N $, let $ c^m $ and $ c^{*m} $ be the state and costate saturations associated to continuous problem \eqref{obj2}-\eqref{sys2} with their discrete counterparts $ c_h^m $ and $ {c}_h^{*m} $, respectively. Then, there exists $C>0$ independent of $h$ and $\Delta t$, such that:
\begin{align}
\nnormh{c^m-{c}_h^m}+\nnormh{c^{*m}-{c}_h^{*m}} \leq C(h+\Delta t). \label{sat1}
\end{align}
\end{theorem}
\begin{proof}
Let $ \tilde{c}_h^n $ be the Riesz projection of $c^{n}$ at time $ t=t^n $ such that
\begin{equation}
A_h(c^n;c^{n}-\tilde{c}_h^{n},z_h)+(b(c^n)\bu^n\cdot \nabla(c^{n}-\tilde{c}_h^{n}),z_h)+\lambda(c^{n}-\tilde{c}_h^{n},z_h)=0, \quad \forall z_h \in M_h, \label{RP1}
\end{equation}
where $ \lambda >0 $ is chosen to guarantee the coercivity of bilinear form defined by \eqref{RP1} with respect to the norm $ \nnormh{\cdot} $.
We then proceed similarly as in  \cite[Lemma 4.2]{kumar12} and
split $ c^n-{c}_h^n=(c^n-\tilde{c}_h^n)+(\tilde{c}_h^n-c_h^n)=\rho^n+\theta^n $, which implies that
\begin{align}
\nnormh{c^n-{c}_h^n}\leq\nnormh{\rho^n}+\nnormh{\theta^n} \leq Ch + \nnormh{\theta^n}.
\end{align}
Testing the state saturation equation in \eqref{sys2} against $ \eta_h z_h $ and integrating over $ \Omega $, we obtain, at $ t=t^{n+1} $
\begin{align}
& (\phi\partial_t c^{n+1}, \eta_hz_h)+A_h(c^{n+1};c^{n+1},z_h)+(b(c^{n+1})\bu^{n+1}\cdot \nabla c^{n+1},\eta_hz_h)\nonumber\\
& \qquad \qquad =(f(c^{n+1})r_0q^{n+1},\eta_hz_h) \label{ifcc3}
\end{align}
Subtracting the discrete state saturation equation  from \eqref{ifcc3}, we then obtain
\begin{align*}
(\phi\partial_t\theta^{n+1}, \eta_hz_h)+A_h(c^{n+1};c^{n+1},z_h)-A_h({c}_h^{n+1};{c}_h^{n+1},z_h)
+(b(c^{n+1})\bu^{n+1}\cdot \nabla c^{n+1},\eta_hz_h)\\-(b({c}_h^{n+1}){\bu}_h^n\cdot \nabla {c}_h^{n+1},\eta_hz_h)= -(\phi\frac{\rho^{n+1}-\rho^{n}}{\Delta t}, \eta_hz_h)-\phi(\partial_t c^{n+1}-\frac{c^{n+1}-c^{n}}{\Delta t}, \eta_hz_h)\\
+(f(c^{n+1})r_0q^{n+1}-f({c}_h^{n+1})r_0q_h^{n+1}, \eta_hz_h)
\end{align*}
Using the definition of $ \epsilon_a $ together with relation \eqref{RP1}, and choosing $ z_h=\partial_t \theta^{n+1} $, we arrive at
\begin{equation}
\begin{split}
& \phi\nnorm{\partial_t\theta^{n+1}}_{\eta_h}^2+A(c_h^{n+1};\theta^{n+1},\partial_t\theta^{n+1})\\
& \quad = -(\phi\frac{\rho^{n+1}-\rho^{n}}{\Delta t}, \eta_h \partial_t\theta^{n+1})-\phi(\partial_t c^{n+1}-\frac{c^{n+1}-c^{n}}{\Delta t}, \eta_h \partial_t\theta^{n+1})\\
&\qquad+(f(c^{n+1})r_0q^{n+1}-f({c}_h^{n+1})r_0q_h^{n+1}, \eta_h \partial_t\theta^{n+1})+(\lambda \rho^{n+1},\eta_h \partial_t\theta^{n+1})\\
&\qquad +[A_h({c}_h^{n+1};\tilde{c}_h^{n+1},\partial_t\theta^{n+1})-A_h(c^{n+1};\tilde{c}_h^{n+1},\partial_t\theta^{n+1})]\\
&\qquad  -(b({c}_h^{n+1}){\bu}_h^n\cdot \nabla {c}_h^{n+1},\partial_t\theta^{n+1}-\eta_h\partial_t\theta^{n+1}) +(b(c^{n+1})\bu^{n+1}\cdot \nabla c^{n+1},\partial_t\theta^{n+1}-\eta_h \partial_t\theta^{n+1})\\
&\qquad -((b(c^{n+1})\bu^{n+1}-b({c}_h^{n+1}){\bu}_h^n)\cdot \nabla \tilde{c}_h^{n+1},\partial_t\theta^{n+1})\\
&\qquad -(b(c^{n+1})\bu^{n+1}\cdot \nabla\theta^{n+1},\eta_h \partial_t\theta^{n+1})
+\epsilon_a(c_h^{n+1};\theta^{n+1},\partial_t\theta^{n+1}). \label{ceq3}
\end{split}
\end{equation}
We can then apply \eqref{epsrel} and the inverse inequality  to obtain
\begin{align}
\epsilon_a(c_h^{n+1};\theta^{n+1},\partial_t \theta^{n+1}) \leq Ch\nnormh{\theta^{n+1}}\nnormh{\partial_t \theta^{n+1}} \leq C\nnormh{\theta^{n+1}}\norm{\partial_t \theta^{n+1}}_{0,\Omega}. \label{epsineq2}
\end{align}
Proceeding similarly as in the proof of Theorem \ref{thm2}, and using \eqref{epsineq2}, we deduce that
the terms in \eqref{ceq3} can be bounded as follows
\begin{equation}
\begin{split}
& \phi\nnorm{\partial_t\theta^{n+1}}_{\eta_h}^2+A(c_h^{n+1};\theta^{n+1},\partial_t\theta^{n+1}) \\
& \quad \leq  C[\norm{c^{n+1}-c_h^{n+1}}_{0,\Omega}^2+\norm{\bu^n-\bu_h^n}_{0,\Omega}^2+\norm{q^{n+1}-q_h^{n+1}}_{L^2(0,T)}+\Delta t\norm{\partial_{tt}c}_{L^2(t_n,t_{n+1}; L^2(\Omega))}^2\\
&\qquad +\Delta t \norm{\bu_t}_{L^2(t_n,t_{n+1}; L^2(\Omega)^2)}^2+\|\rho^{n+1}\|_{0,\Om}^2
+(\Delta t)^{-1}\norm{\partial_t \rho}_{L^2(t_n,t_{n+1}; L^2(\Omega))}^2 \\
&\qquad +\nnormh{\theta^{n+1}}^2  +\|\partial_t\theta^{n+1}\|_{0,\Om}^2],
\end{split} \label{csbd1}
\end{equation}
and therefore it can be seen that
\begin{align}
\tilde{A}_h(c_h^{n+1};\theta^{n+1},\partial_t\theta^{n+1}) \geq \frac{1}{2\Delta t}\left[\tilde{A}_h(c_h^{n+1};\theta^{n+1},\theta^{n+1})-\tilde{A}_h(c_h^{n+1};\theta^{n},\theta^{n}) \right]. \label{tsieq1}
\end{align}
Summing over $ n=0,\ldots,m-1 $, using the equivalence between the norms $ \norm{\cdot}_{\eta_h} $ and $ \norm{\cdot}_{0,\Omega} $, the coercivity of the bilinear form $ \tilde{A}_h(c_h^{n+1},\cdot,\cdot)$ and noting that $ \theta^0=0 $ in \eqref{csbd1};
we get that
\begin{align*}
\nnormh{\theta^{m}}^2 \leq & C \Delta t \sum\limits_{n=0}^{m-1}[\norm{c^{n+1}-c_h^{n+1}}_{0,\Omega}^2+\norm{\bu^n-\bu_h^n}_{0,\Omega}^2+\norm{q^{n+1}-q_h^{n+1}}_{L^2(0,T)}\\
& \quad +\Delta t\norm{\partial_{tt}c}_{L^2(t_n,t_{n+1}; L^2(\Omega))}^2
+\Delta t \norm{\bu_t}_{L^2(t_n,t_{n+1}; L^2(\Omega)^2)}^2+\|\rho^{n+1}\|_{0,\Om}^2\\
&\quad +(\Delta t)^{-1}\norm{\partial_t \rho}_{L^2(t_n,t_{n+1}; L^2(\Omega))}^2 \!+\nnormh{\theta^{n+1}}^2],
\end{align*}
 for an appropriate value of  the constant $ C $.
Applying the discrete Gronwall's lemma and the estimates in Theorem \ref{thm3}, leads to the bound
 $ \nnormh{\theta^{m}} \leq C(h+\Delta t) $, which together with \eqref{sat1}, implies that
\begin{align*}
\nnormh{c^{m}-c_h^m} \leq C(h+\Delta t).
\end{align*}
The bound for  $ \nnormh{c^{*m}-c_h^{*m}}$ can be derived using the same approach.
\end{proof}

\section{Implementation of the optimal control solver}\label{sec:impl}
Now we proceed to describe the implementation of the numerical methods discussed in Section \ref{sec:fve}.
For the specific applications in the present context, it is known that the pressure field exhibits much smoother
profiles in time, compared to the evolution of saturation.
We will therefore consider a first partition of $J$ as $ 0=t_0<t_1<\cdots < t_M=T $
with step length $ \Delta t_m=t_{m+1}-t_m $ dedicated for the Darcy equations, whereas for the saturation
equation we take $ 0=t^0<t^1<\cdots< t_N=T $ with timestep $ \Delta t^n=t^{n+1}-t^n $. We remark that
such a splitting will still produce accurate approximations (see the discussion in e.g. \cite{ewing82}).

\paragraph{A splitting method for both state and costate problems.}
To lighten the notation we will adopt the following notation
\begin{align*}
& C^n = c_h(t^n), \quad C_m = c_h(t_m), \quad C^{*n} = c_h^*(t^n), \quad
C_m^* = c_h^*(t_m), \\
& \bU_m = \bu_h(t_m), \quad P_m = p_h(t_m), \quad
\bU_m^* =  \bu_h^*(t_m), \quad P_m^* =  p_h^*(t_m).
\end{align*}
In addition, if $ t_{m-1}<t^n\leq t_m $, then velocity approximation at $ t=t^n $ is defined by
\begin{align*}
\bU^n &=\left( 1+ \frac{t^n-t_{m-1}}{\Delta t_{m-2}} \right)\bU_{m-1}- \frac{t^n-t_{m-1}}{\Delta t_{m-2}} \bU_{m-2},  \text{ for } m=2,\ldots,M,\quad
\bU^n =\bU_0, \quad \text{for}\,\, m=1,\\
\bU^{*n} &=\left( 1+ \frac{t^n-t_{m-1}}{\Delta t_{m-2}} \right)\bU^*_{m}- \frac{t^n-t_{m-1}}{\Delta t_{m-2}} \bU^*_{m-1}, \text{ for } m=M-1,\ldots,1,
\quad \bU^{n*} =\bU_M, \text{ for } m=M.
\end{align*}
We then rewrite the discrete  state Darcy equations
\eqref{sdvp}-\eqref{sdp} is to find $(\bU,P):\lbrace t_0,\ldots,t_M \rbrace \rightarrow U_h \times W_h $ such that
\begin{equation}
\begin{split}
(\alpha(C_m)\bU_m,\gamma_h \bv_h)-( \nabla\cdot \bv_h,P_m) &=0 \quad \forall \bv_h \in U_h,\\
(\nabla \cdot \bU_m,w_h)  - ((r_0-r_1)q_h^m,w_h) &= 0 \quad \forall w_h \in W_h.
\end{split} \label{dpveq1}
\end{equation}
On the other hand, assuming a backward difference approximation of the first order time derivative,
the discrete state saturation equation \eqref{sdc} reduces to find
$ C:\lbrace t^0,\ldots,t^N \rbrace \rightarrow M_h $ such that
\begin{equation}
(\phi\frac{C^{n+1}-C^n}{\Delta t^n}, \eta_hz_h)+A_h(C^{n+1};C^{n+1},z_h)+(b(C^{n+1})\bU^{n+1}\cdot\nabla C^{n+1},\eta_hz_h)=(f(C^{n+1})r_0 q_h^{i+1}, \eta_hz_h). \label{dsateq1}
\end{equation}

Next, for a given control $ q_h^0 $, we take $ C^0=C_0=c_{0,h} $ and obtain velocity and pressure
approximations $ (\bU_0,P_0) $ from \eqref{dpveq1}.  Using $ \bU_0 $ we can compute $ C^1 $ from
\eqref{dsateq1}, and repeat the process throughout the time horizon.
Then the discrete costate Darcy problem \eqref{sdavp}-\eqref{sdap} consists in finding
$ (\bU^*,P^*):\lbrace t_M, \ldots, t_0 \rbrace \rightarrow U_h \times W_h $ such that
\begin{equation}
 \begin{split}
(\alpha(C_m)\bU^*_m,\gamma_h \bv_h)-(\nabla\cdot \bv_h,P^{*}_m) & =-(C^{*}_m b(C_m)\nabla C_m,\gamma_h \bv_h), \\
(\nabla \cdot\bU^{*m},w_h) & =0.
\end{split} \label{dpveq2}
\end{equation}
The discrete costate saturation equation \eqref{sdac} reads: Find $ C^*:\lbrace t^N, \ldots, t^0 \rbrace \rightarrow M_h $ such that
\begin{equation}
 \begin{aligned}
-(\phi\frac{C^{*(n+1)}-C^{*n}}{\Delta t^n}, \eta_hz_h)+A_h(C^{n+1};C^{*(n+1)},z_h)-(b(C^{n+1})\bU^{n+1}\cdot\nabla C^{*(n+1)}, \eta_hz_h) \qquad & \\+(\mathcal{D}'(C^{n+1})\nabla C^{n+1}\cdot \nabla C^{*(n+1)}, \eta_hz_h)+(\alpha'(C^{n+1})\bU^{*(n+1)}\cdot \bU^n,\eta_hz_h) \qquad & \\+(r_1q_h^{n+1}b(C^{n+1})C^{*(n+1)},\eta_hz_h) =(wC^{n+1}, \eta_hz_h). &
\end{aligned}
\label{dsateq2}
\end{equation}
Using $ C^{*N}=C_{*N}=0 $ we find $ (\bU^*_N,P^*_N) $ from \eqref{dpveq2} and using $ \bU^*_N $ we obtain
$ C^{N-1} $ from \eqref{dsateq2}. The process is then repeated down to $t=0$.

\paragraph{Discrete problems in matrix form.}
Let $ \lbrace \bPhi_i \rbrace_{j=1}^{N_m} $ be
basis functions for the trial space $ U_h $ and $ \lbrace \chi_l^* \rbrace_{l=1}^{N_e}  $
denote characteristic functions for each element in $\cT_h$, which form basis functions for $ W_h $.
We denote by $ N_m $ the number of of midpoints of the edges  in $ \cT_h $, and $ N_e $ stands for
the total number of elements. The vectors containing the unknowns for each variable are then
constructed as
$$ \bU_m=\sum_{j=1}^{N_m} \alpha_j^m \bPhi_j,\quad P_m=\sum\limits_{l=1}^{N_e} \beta_l^m \chi_l^*,\quad \bU_m^*=\sum\limits_{j=1}^{N_m} \alpha_j^{*m} \bPhi_j, \quad P_m^*=\sum\limits_{l=1}^{N_e} \beta_l^{*m} \chi_l^*  ,$$
where the coefficients are specified as
$$ \alpha_j=(\bu_h \cdot \bn_j)(M_j),\quad \beta_l=p_h(b_{Kl}), \quad\alpha_j^*=(\bu_h^* \cdot \bn_j)(M_j),\quad \beta_l^*=p_h^*(b_{Kl}) ,$$
with $ b_{Kl} $  denoting the barycentre of the triangle $ K_l $.
After defining the following matrix and vector entries (with indexes $1\leq l \leq N_e, 1\leq i,j\leq N_m$)
\begin{align*}
(A_m)_{ij}:=\int_{T^*_{M_i}} \alpha(C_m)\bPhi_j\cdot \bPhi_i(M_i)\dx,\quad
(B_m)_{lj}:=\int_{T_l} \nabla \cdot \bPhi_j \dx, \\
(F_m)_l :=\int_{T_l} (r_0-r_1)q_h^m \dx, \quad
(F_m^*)_i:=-\int_{T^*_{M_i}} C^{*}_m b(C_m)\nabla C_m \cdot \bPhi_i(M_i)\dx,
\end{align*}
we can write the matrix form of the discrete state Darcy equations \eqref{dpveq1} as
\begin{equation}
\begin{pmatrix}
\bA_m & \bB_m \\
\bB_m^T &  \boldsymbol{0}
\end{pmatrix}
\begin{pmatrix}
\balpha^m\\
\bbeta^m
\end{pmatrix}=\begin{pmatrix}
\boldsymbol{0} \\
\bF_m
\end{pmatrix}, \label{matPV1}
\end{equation}
and the discrete costate Darcy problem \eqref{dpveq2} in matrix form as
\begin{equation}
\begin{pmatrix}
\bA_m & \bB_m \\
\bB_m^T &  \boldsymbol{0}
\end{pmatrix}
\begin{pmatrix}
\balpha^{*m}\\
\bbeta^{*m}
\end{pmatrix}=\begin{pmatrix}
\bF_m^*\\
\boldsymbol{0}
\end{pmatrix}. \label{matPV2}
\end{equation}

Regarding the transport equation, let $ \lbrace \Psi_i \rbrace_{i=1}^{N_h} $ denote a basis for $ M_h $, so that
the vectors of state and costate saturations are respectively  $ C^n=\sum_{i=1}^{N_h} \delta_i^n \Psi_i $ and $ C^{*n}=\sum_{i=1}^{N_h} \delta_i^{*n} \Psi_i $. We use the notation $\bdelta^n =(C^n(P_i))_{i=1}^{N_h}$ and $\bdelta^{*n} =(C^{*n}(P_i))_{i=1}^{N_h}$, and
define the following matrix and vector entries (with $1\leq i,j \leq N_h$)
\begin{gather*}
(D^n)_{ij}:=\int_{K_i^*} \! \Psi_i \eta_h\Psi_j \dx,\,
(E_n)_{ij}:=\int_{K_i^*}\!\!(b(C^n)\bU^n \cdot \nabla \Psi_i)\eta_h\Psi_j \dx, \,
(G_n)_i := \int_{K_i^*} f(C^n)r_0 q_h^n \eta_h\Psi_i \dx,\\
(R_n)_{ij}:=\int_{K_i^*} r_1 q_h^n b(C^n)\Psi_i\eta_h\Psi_j \dx, \,
(S^n)_{ij}:=\int_{K_i^*} \mathcal{D}'(C^n)\nabla C^n \cdot \nabla\Psi_i \eta_h\Psi_j \dx,\\
(W_n)_i:=\int_{K_i^*} w C^n \eta_h\Psi_i,\,
(Z_n)_i:=\int_{K_i^*} \alpha'(C^n) \bU^n\cdot \bU^{*n}\eta_h\Psi_i \dx, \quad H_n: =T_1^n+T_2^n+T_3^n+T_4^n, \\
(T_1^n)_{ij}=-\sum_{ K\in\cT_h}\sum_{k=1}^3\int\limits_{v_{k+1}b_Kv_k} \!\!\!\!\mathcal{D}(C^n)\nabla \Psi_i\cdot \bn \eta_h\Psi_j \ds,\
(T_2^n)_{ij}= -\sum_{e\in \mathcal{E}_h}\int_e \jump{\eta_h\Psi_i}\cdot \langle \mathcal{D}(C^n)\nabla \Psi_j \rangle \ds,\\
(T_3^n)_{ij}=-\sum_{e\in \mathcal{E}_h}\int_e \jump{\eta_h\Psi_j}\cdot \langle \mathcal{D}(C^n)\nabla \Psi_i \rangle \ds,\quad
(T_4^n)_{ij}= \sum_{e\in \mathcal{E}_h}\int_e \frac{\xi}{h_e}\jump{\Psi_i}\jump{\Psi_j}\ds.
\end{gather*}
where $v_k$ denotes a vertex of $K$.

Therefore the state saturation equation \eqref{dsateq1} adopts the following matrix form
\begin{align}
\left[\phi \bD^n+\Delta t_n(\bE_n+\bH_n) \right] \bdelta^{n+1}=\phi \bD^n \bdelta^n+ \Delta t_n \bG^n, \label{matS1}
\end{align}
and likewise, the matrix form of the costate saturation equation \eqref{dsateq2} reads
\begin{align}
-\phi \bD^n \bdelta^{*n}=\left[-\phi \bD^n+\Delta t_n(-\bE_n+\bH_n+\bS_n+\bR_n) \right] \bdelta^{*(n+1)}-\Delta t_n (-\bZ_n+\bW_n). \label{matS2}
\end{align}



\paragraph{Active set strategy.}
The control constraints can be implemented following the active set strategy adapted from \cite{hinze05,kumar16},
where the main steps of the method are be summarised in Algorithm \ref{algo1}, below.

\noindent We first notice that the discrete variational inequality
\begin{align*}
(f(C^n)r_0 C^{*n}-(r_0-r_1)P^{*n}+\alpha_0 q_h^n, \tilde{q_h}-q_h^n) \geq 0, \quad \forall \tilde{q_h} \in \Qad,\,\,n=0,\ldots,N,
\end{align*}
is equivalently written as
\begin{align*}
q_h^n :=\max \lbrace 0, \min \lbrace \tilde{q}, -\alpha_0^{-1}\int_{\Omega} f(C^n)r_0 C^{*n}-(r_0-r_1)P^{*n} \dx \rbrace \rbrace,\,\,n=0,\ldots,N,
\end{align*}
(see e.g. \cite{chang14}), and
we observe that the quantity $ -\alpha_0^{-1}\int_{\Omega} f(C^n)r_0 C^{*n}-(r_0-r_1)P^{*n} \dx  $ can be considered as a measure for the activity of control constraints.
For each time horizon, we proceed to define the active sets $ A_{k+1}^{-,n} $  and $ A_{k+1}^{+,n} $ as well as inactive set $ I_{k+1}^n $, at the current iteration, as follows
\begin{align*}
 A_{k+1}^{-,n} &:=\left\lbrace x\in \Om: -\alpha_0^{-1}\int_{\Omega} f(C^n_k)r_0 C^{*n}_k-(r_0-r_1)P^{*n}_k \dx < 0 \right\rbrace,\, n=0,\ldots,N, \\
 A_{k+1}^{+,n} &:=\left\lbrace x\in \Om: -\alpha_0^{-1}\int_{\Omega} f(C^n_k)r_0 C^{*n}_k-(r_0-r_1)P^{*n}_k \dx > \tilde{q} \right\rbrace,\, n=0,\ldots,N,\\
 I_{k+1}^n &:=\Om \setminus ( A_{k+1}^{-,n} \cup  A_{k+1}^{+,n}),
\end{align*}
then we have that
\begin{align*}
q_{h,k+1}^n=\begin{cases}
    0 & \text{on}\quad A_{k+1}^{-,n} ,\\
    -\alpha_0^{-1}\int_{\Omega} f(C^n_k)r_0 C^{*n}_k-(r_0-r_1)P^{*n}_k  & \text{on}\quad  I_{k+1}^n,\\
  \tilde{q}  &  \text{on} \quad A_{k+1}^{+,n},
\end{cases}
\end{align*}
or, equivalently,
\begin{align}
q_{h,k+1}^n=\tilde{q}\chi_{A_{k+1}^{+,n}}-\alpha_0^{-1}\int_{\Omega} f(C^n_k)r_0 C^{*n}_k-(r_0-r_1)P^{*n}_k(1-\chi_{A_{k+1}^{-,n}}-\chi_{A_{k+1}^{+,n}}), \label{qatrel}
\end{align}
where $ \chi_{A_{k+1}^{-,n}} $ and $ \chi_{A_{k+1}^{+,n}} $ are the characteristic functions corresponding to the active sets $ A_{k+1}^{-,n} $ and $ A_{k+1}^{+,n} $, respectively.
Using the value of $ C, C^{*} $ and $ P^{*} $, we can compute the discrete control $ q_h $ for each time horizon. We can then repeat the process until we reach the termination criteria, that is, when two successive active sets coincide.

\begin{algorithm}[H]
\algsetup{indent=2em}
\caption{Method of active sets}
\begin{algorithmic}[1] \label{algo1}
\STATE Choose \AND store arbitrary initial guess $ q_{h,0} $ \AND set $ k=0$
\FOR {$k=0,1,\ldots,$}
\STATE Given the control $ q_{h,k} $, \textbf{compute} $ (\bU_k,P_k):\lbrace t_0,t_1,\ldots,t_M \rbrace \rightarrow U_h \times W_h $ from \eqref{matPV1}
\STATE \textbf{compute} $ C_k:\lbrace t^0,t^1,\ldots,t^N \rbrace \rightarrow M_h $ from \eqref{matS1}
\STATE \textbf{compute} $ (\bU^{*}_k,P^{*}_k):\lbrace t_M, t_{M-1},\ldots,t_0 \rbrace \rightarrow U_h \times W_h $ from \eqref{matPV2}
\STATE \textbf{compute} $ C^{*}_k: \lbrace t^N, t^{N-1},\ldots,t^0 \rbrace \rightarrow M_h $ from \eqref{matS2}
\STATE Update $ q_{h,k} \leftarrow q_{h,k+1}$ from relation \eqref{qatrel}
 \IF{ $ A_{k+1}^{-}=A_{k}^{-} $ \AND $ A_{k+1}^{+}=A_{k}^{+} $}
 \STATE \textbf{stop}
 \ELSE
 \STATE \textbf{go to} step 3
 \ENDIF
\ENDFOR
\end{algorithmic}
\end{algorithm}

\end{document}